\documentclass[11pt]{article}
\usepackage{amssymb,amsthm,amsmath}

\usepackage{palatino}
\usepackage{rotating}
\usepackage{hyperref}
\usepackage{mdwlist}
\usepackage{tikz}
\usetikzlibrary{positioning}
\usepackage{subfig}
\usepackage{enumerate}
\usepackage{indentfirst}

\newtheorem{thm}{Theorem}[section]
\newtheorem*{thm*}{Theorem}

\newtheorem*{lemma*}{Lemma}
\newtheorem{claim}[thm]{Claim}
\newtheorem*{claim*}{Claim}

\theoremstyle{definition}

\newtheorem*{define*}{Definition}
\newtheorem*{note*}{Notation}

\renewenvironment{proof}[1][]{\begin{trivlist}
\item[\hspace{\labelsep}{\bf\noindent Proof#1.\/}] }{\qed\end{trivlist}}

\newcommand{\remove}[1]{}

\renewcommand{\r}{\tilde{r}}
\renewcommand{\a}{\tilde{\alpha}}
\newcommand{\bs}{\setminus}

\title{Radon Numbers for Trees}

\author{Shoham Letzter
\thanks{Department of Pure Mathematics and Mathematical Statistics, Centre for Mathematical Sciences, Wilberforce
Road, Cambridge, CB3 0WB, UK. Email: s.letzter@dpmms.cam.ac.uk }
}

\begin{document}

\maketitle

\begin{abstract}
\setlength{\parindent}{0in} 
\setlength{\parskip}{.08in} 
\noindent
Many interesting problems are obtained by attempting to generalize classical results on convexity in Euclidean spaces to other convexity spaces, in particular to convexity spaces on graphs. 
In this paper we consider $P_3$-convexity on graphs. A set $U$ of vertices in a graph $G$ is \emph{$P_3$-convex} if every vertex not in $U$ has at most one neighbour in $U$.
More specifically, we consider Radon numbers for $P_3$-convexity in trees.

Tverberg's theorem states that every set of $(k-1)(d+1)-1$ points in $\mathbb{R}^d$ can be partitioned into $k$ sets with intersecting convex hulls. As a special case of Eckhoff's conjecture, we show that a similar result holds for $P_3$-convexity in trees.

A set $U$ of vertices in a graph $G$ is \emph{free}, if no vertex of $G$ has more than one neighbour in $U$.
We prove an inequality relating the Radon number for $P_3$-convexity in trees with the size of a maximal free set.  
\setlength{\parskip}{.1in} 
\end{abstract}

\section{Introduction}

Radon's classical lemma \cite{classical_radon} states that every set of $d+2$ points in $\mathbb{R}^d$ can be partitioned into two sets whose convex hulls intersect. Tverberg \cite{tverberg} generalized this to partitions into more than two sets. Namely, every set of at least $(k-1)(d+1)+1$  points in $\mathbb{R}^d$ can be partitioned into $k$ sets whose convex hulls have a point in common.

Inspired by this, Eckhoff  conjectured in \cite{part_conj} that the situation is similar in general convexity spaces.
A \emph{convexity space} is a pair $(X,\mathcal{C})$ where $X$ is a set and $\mathcal{C}$ is a collection of subsets of $X$, called \emph{convex sets}, such that $\emptyset$ and $X$ are convex and the intersection of convex sets is convex.
The \emph{convex hull} of a set $S\subseteq X$, denoted by $H_\mathcal{C}(S)$, is the minimal convex set containing $S$, i.e.~the intersection of all convex sets containing $S$.
For a set $S$, a $k$-Radon partition is a partition of $S$ into $k$ sets whose convex hulls have a point in common.
A set is \emph{$k$-anti Radon} (or $k$-a.r.) if it has no $k$-Radon partition.
The $k$th Radon number of $(X,\mathcal{C})$ is the minimal number (if it exists) $r_k(\mathcal{C})$ such that every set $S\subseteq X$ of size at least $r_k(\mathcal{C})$ has a $k$-Radon partition.
Eckhoff \cite{part_conj} conjectured that $r_k(\mathcal{C})\le (k-1)(r_2(\mathcal{C})-1)+1$ holds in every convexity space.
This conjecture has been proved in several convexity spaces including trees with geodesic convexity \cite{jamison}. However, the general conjecture has recently been disproved by Bukh \cite{bukh}.

The notion of a $k$-a.r~set can be generalized to multi-sets by considering partitions of multi-sets rather than sets. In this paper, we define  $\r_k(\mathcal{C})$ to be the size of the largest $k$-a.r.~multi-set. Note that $\r_k(\mathcal{C})\ge r_k(\mathcal{C})-1$, with equality for $k=2$ (as a $2$-a.r.~multi-set is a set, i.e.~no element can appear more than once).
When $k=2$, we often omit the prefix $k$, e.g.~a $2$-a.r.~set may be called an a.r.~set and $\r(\mathcal{C})=\r_2(\mathcal{C})$.

In this paper we shall study $P_3$-convexity in trees. For a graph $G$, a set $U$ of vertices of $G$ is \emph{$P_3$-convex} or, briefly, \emph{convex} if every vertex not in $U$ has at most one neighbour in $U$. Equivalently, $U$ is convex if it contains all middle vertices in the paths of length $2$  between two vertices of $U$.
$P_3$-convexity was first considered in the context of directed graphs and tournaments (see \cite{erdos},\cite{erdos2},\cite{moon},\cite{Varlet}).

Throughout this paper graphs are always finite, simple and undirected.
For a graph $G$, let $\r_k(G)$ denote the $k$th Radon number for $P_3$-convexity on $G$, and for a set $U\subseteq V(G)$, let $H_G(U)$ denote the convex hull of $U$ in $G$.

As the first main result of our paper, we show that Eckhoff's conjecture holds for $P_3$-convexity on trees.
\begin{thm}
\label{partition}
Let $T$ be a tree, $k\ge 3$. Then $\tilde{r}_k(T)\le (k-1)\tilde{r}_2(T)$.
\end{thm}

Given a graph $G$,  call a set $A\subseteq V(G)$ \emph{free} if every vertex of $G$ has at most one neighbour in $A$. Note that every free set in a graph $G$ is also convex and the converse does not hold in general.
Let $\tilde{\alpha}(G)$ be the size of the largest free set in $G$.
It follows that $\tilde{r}(G)\ge\tilde{\alpha}(G)$.
Our second main theorem answers a question posed by Dourado et al. \cite{radon}.
\begin{thm}\label{free}
Let $T$ be a tree. Then 
$\r_2(T)\le 2\a(T)$.
\end{thm}
We shall show that this theorem is sharp in the sense that there are infinitely many trees for which we have equality.

The last inequality is not true in general as shown by the graph $G_1$ in Figure \ref{fig_ex_free}.
Every two of the seven vertices of $G_1$ have a common neighbour, hence $\a(G_1)=1$.
It is easy to check that the set $A=\{2,4,6\}$ of vertices of $G_1$ is a.r.~and that every set of $4$ vertices of $G_1$ is not a.r., therefore $\r(G_1)=3$.

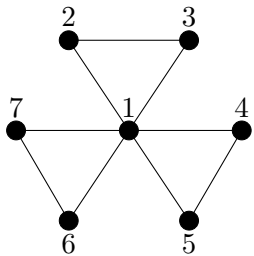
\begin{figure}[ht]\centering
\begin{tikzpicture}
\tikzstyle{bel}=[below=0.05]
\draw[fill](0,0) circle [radius=0.125];
\draw[fill](1.5,0) circle [radius=0.125];
\draw[fill](-1.5,0) circle [radius=0.125];
\draw[fill](-.8,1.2) circle [radius=0.125];
\draw[fill](.8,1.2) circle [radius=0.125];
\draw[fill](.8,-1.2) circle [radius=0.125];
\draw[fill](-.8,-1.2) circle [radius=0.125];

\draw(0,0)--(.8,1.2)--(-.8,1.2)--(0,0)--(1.5,0)--(.8,-1.2)--(0,0)--(-1.5,0)--(-.8,-1.2)--(0,0); 

\node[above=0.05] at (0,0) {$1$};
\node[above=0.05] at (-.8,1.2) {$2$};
\node[above=0.05] at (.8,1.2) {$3$};
\node[above=0.05] at (1.5,0) {$4$};
\node[bel] at (.8,-1.2) {$5$};
\node[bel] at (-.8,-1.2) {$6$};
\node[above=0.05] at (-1.5,0) {$7$};

\end{tikzpicture}
\caption{Graph $G_1$}
\label{fig_ex_free}
\end{figure}

We prove Theorem \ref{partition} in Section \ref{section_partition} and Theorem \ref{free} in Section \ref{section_free}.

\section{Proof of Theorem \ref{partition}}\label{section_partition}
Before proving this theorem, we introduce some notation.
For a graph $G$ and a vertex $v\in V(G)$, define
\begin{equation}\label{defn_radon*}
\r_k^*(G,v)=\max\{|R|\!: R\text{ is a }k\text{-a.r.~multi-set and }v\notin H_G(R) \}.
\end{equation}

\begin{proof}[ of Theorem \ref{partition}]
We shall prove more than claimed in the statement of the theorem. Namely, we shall show that for every tree $T$
\begin{itemize}
\item
$\r_k^*(T,v)\le (k-1)\r_2^*(T,v)$ holds for every $v\in V(T)$,
\item
$\r_k(T)\le (k-1)\r_2(T)$.
\end{itemize}
Our proof is by induction on $n=|V(T)|$.
Both statements are clear for $n\le 2$.

Let $T$ be a tree with $n\ge 3$ vertices.
The second statement follows easily by induction using expression \ref{eqn_radon} for $\r_k^*$ below.
For a vertex $v\in V(T)$, let $v_1,\ldots,v_l$ be its neighbours, and for every $i\in [l]$ let $T_i$ be the component of $v_i$ in $T\setminus\{v\}$. Then
\begin{equation}\label{eqn_radon}
\r_k^*(T,v)=\max_{i\in [l]}(\sum\limits_{j\neq i}\r_k^*(T_j,v_j)+\r_k(T_i)).
\end{equation} 

We now prove that $\r_k(T)\le (k-1)\r_2(T)$.
Let $R$ be a $k$-a.r.~multi-set of maximal size.
If T has endvertex $v$ which is not in $R$, let $T'=T\setminus \{v\}$. Then by induction,
\begin{equation*}
|R|=\r_k(T')\le (k-1)r_2(T')\le (k-1)r_2(T).
\end{equation*}
Thus we may assume that $R$ contains each endvertex of $T$ at least once.

In the rest of the proof we consider two possible cases which will be dealt with in different subsections.
\subsection{Case 1.}

There is a longest path $v_1,\ldots,v_m$ in $T$ such that $\deg(v_2)\ge 3$.
\\

Let $z=v_3$, $y=v_2$ and $x_1,\ldots,x_l$ be the neighbours of $y$ other than $z$. Note $l\ge 2$, and by the choice of $v_1,\ldots,v_m$ as a longest path, $x_1,\ldots ,x_l$ are all endvertices (see Figure \ref{fig_radon_1}).

\begin{figure}[ht]\centering
\begin{tikzpicture}
\tikzstyle{bel}=[below=0.05]

\draw[fill](0,0) circle [radius=0.125];
\draw[fill](1,0) circle [radius=0.125];
\draw[fill](-1,1.5) circle [radius=0.125];
\draw[fill](-1,0.5) circle [radius=0.125];
\draw[fill](-1,-1.5) circle [radius=0.125];

\draw (-1,.5)--(0,0)--(-1,1.5);
\draw (0,0)--(1,0)--(1.8,0.4);
\draw (1,0)--(1.8,-.7);
\draw (-1,-1.5)--(0,0);

\draw [fill](-1,0.15) circle [radius=0.025];
\draw [fill](-1,-0.1) circle [radius=0.025];
\draw [fill](-1,-0.35) circle [radius=0.025];
\draw [fill](-1,-0.6) circle [radius=0.025];
\draw [fill](-1,-0.85) circle [radius=0.025];
\draw [fill](-1,-1.1) circle [radius=0.025];

\node[left] at (-1,1.5){$x_1$};
\node[left] at (-1,.5){$x_2$};
\node[left] at (-1,-1.5){$x_l$};

\node [bel] at (0,0){$y$};
\node [bel] at (1,0){$z$};d

\boldmath
\node at (0.1,0.8){$T'$};
\node at (1.1,0.6){$T''$};
\draw [rounded corners=25pt](2,-1.2)--(-.5,-1.2)--(-.5,1.2)--(2,1.2);
\draw [rounded corners=20pt](2,-.9)--(.5,-.9)--(.5,.9)--(2,.9);

\end{tikzpicture}
\caption{Case 1}
\label{fig_radon_1}
\end{figure}

Denote by $s_i$, $i\in [l]$, the number of appearances of $x_i$ in $R$ and by $t$ the number of appearances of $y$ in $R$. By our assumption that $R$ contains every endvertex at least once, $s_i\ge 1$ for every $i\in [l]$.
As $R$ is $k$-a.r., $s_i,t\le k-1$.
Let $s=s_1+\ldots +s_l$.
We consider three cases according to the value of $s$.

\begin{enumerate}[(a)]
\item\label{case1a_partition}
$s\le 2k-2$.\\
Let $\tau=\min\{s,2k-2-s\}$, $\sigma=(s-\tau)/2$.
Note that $\sigma$ is an integer and $\tau+\sigma\le k-1$.
Set $T'=T\setminus\{x_1,\ldots ,x_l\}$ (see Figure \ref{fig_radon_1}). Let $R'$ be the multi-set obtained by adding $\sigma$ copies of $y$ to $R\cap V(T')$. Note $|R|=|R'|+s-\sigma=|R'|+\sigma+\tau\le |R'|+ k-1$.

\begin{claim}\label{pairing}
$R'$ is $k$-a.r..
\end{claim}
\begin{proof}
We shall show the existence of sequences $a_1,\ldots ,a_{\sigma}$, $b_1,\ldots,b_{\sigma}$ satisfying
\begin{itemize*}
\item
$a_j,b_j\in [l]$ and $a_j\neq b_j$ for every $j\in[\sigma]$,
\item
$|\{j\in[\sigma]\!: a_j=i \}| + |\{j\in[\sigma]\!: b_j=i \}|=s_i$ for every $i\in [l]$.
\end{itemize*}
Note that the existence of such sequences completes the proof of this claim.
Assume to the contrary that $R'=R_1'\cup \ldots \cup R_k'$ is a $k$-Radon partition of $R'$.
Obtain $R_1,\ldots,R_k$ by replacing each of the $\sigma$ new copies of $y$ with a distinct pair $x_{a_j},x_{b_j}$, $j\in [\sigma]$.
By the choice of the $a_j$'s and $b_j$'s, $R=R_1\cup \ldots \cup R_k$ is a partition of $R$. 
Clearly $H_T(R_l)\cap V(T')=H_{T'}(R_l')$ for every $l\in[k]$, hence this partition is a $k$-Radon partition of $R$, contrary to the choice of $R$ as a $k$-a.r.~set.

It thus remains to show the existence of such sequences.
By induction on $k$, we show that if $s=s_1+\ldots+s_l\le 2k-2$ and $s_i\le k-1$ for every $i\in [l]$ we can find two sequences satisfying the above.
We assume $\sigma\ge 1$ or equivalently $s\ge k$, because otherwise there is nothing to prove.
When $k=2$ we thus have that without loss of generality $s_1=s_2=1$, and we set $a_1=1, b_1=2$.
If $k\ge 3$ assume that $s_1\ge s_2\ge \ldots \ge s_l$, and let $a_\sigma =1,b_\sigma =2$.
Now set
\begin{equation*}
s_i'=\left\{
\begin{array}{ll}
s_i-1 & i\in\{1,2\} \\
s_i & \text{otherwise}
\end{array} \right..
\end{equation*}
Note that $s'=s_1'+\ldots s_l'\le 2k-4$ and $s_i'\le k-2$ (otherwise $s_3\ge k-1$ and $s_1+s_2+s_3\ge 3(k-1)>2(k-1)$, a contradiction). Also $\sigma'=2k-4-s'=\sigma-1$. We can now continue by induction.
\end{proof}

Using Claim \ref{pairing} we conclude by induction that
\begin{equation*}
\r_k(T)=|R|\le|R'|+(k-1)\le \r_k(T')+(k-1)\le (k-1)(\r_2(T')+1).
\end{equation*}

The following claim completes the proof of Theorem \ref{partition} in Case 1\ref{case1a_partition}.
\begin{claim}\label{adding_leafs}
$\r_2(T)\ge r_2(T')+1$.
\end{claim}
\begin{proof}
Let $S'$ be an a.r.~set in $T'$ of maximal size.
Set 
\begin{equation*}
S=\left\{
\begin{array}{ll}
S'\cup\{x_1\} & y\notin S' \\
(S'\setminus\{y\})\cup\{x_1,x_2\} & y\in S'
\end{array} \right..
\end{equation*}
Note that $|S|=|S'|+1$. We shall show that $S$ is a.r, thus proving the claim.
Assume to the contrary that there exists a partition $S=A\cup B$ with $H_T(A)\cap H_T(B)\neq\emptyset$.
We assume $x_1\in A$.

Consider the following three possibilities.
\begin{itemize*}
\item
$y\notin S'$.\\
Set $A'=A\setminus\{x_1\}$.
Note
\begin{itemize*}
\item
$S'=A'\cup B$ is a partition of $S'$.\\
\item
$ 
H_T(A)=\left\{
\begin{array}{ll}
H_{T'}(A')\cup \{x_1\}&z\notin H_{T'}(A')\\
H_{T'}(A')\cup\{x_1,y\}&z\in H_{T'}(A')
\end{array}\right..
$\\
\item
$H_T(B)=H_{T'}(B)$, and $y\notin H_T(B)$.
\end{itemize*}

Therefore
$H_T(A)\cap H_T(B)=H_{T'}(A')\cap H_{T'}(B)=\emptyset$, a contradiction.
\item 
$y\in S'$ and $x_1,x_2\in A$.\\
Let $A'=(A\setminus\{x_1,x_2\})\cup\{y\}$.
Then $H_T(A)=H_{T'}(A')\cup \{x_1,x_2\}$, and $H_T(B)=H_{T'}(B)$.
As before we reach a contradiction.

\item
$y\in S'$ and $x_1\in A$, $x_2\in B$.\\
As $S'$ is a.r.~and $y\in S'$, $z\notin H_{T'}(A\setminus\{x_1\})\cap H_{T'}(B\setminus\{x_2\})$.
Without loss of generality, $z\notin H_{T'}(B\setminus\{x_2\})$.
Set $A'=(A\setminus\{x_1\})\cup \{y\}$.
As before $H_T(A)\subseteq H_{T'}(A')\cup \{x_1,y\}$ and $H_T(B)=H_{T'}(B)\cup \{x_2\}$. This leads to a contradiction to $S'$ being a.r..
\end{itemize*}
\end{proof}

\item\label{case1b_partition}
$s=2k-1$.\\
Define $T'$ as before, and let $R'$ be the union of $R\cap V(T')$ with a copy of $x_1$ and $k-1$ copies of $y$.
\begin{claim}
$R'$ is $k$-a.r..
\end{claim}
\begin{proof}
Replacing $s_1$ by $s_1-1$ returns us to the setting of Claim \ref{pairing}. Following the same arguments we obtain this claim.
\end{proof}
Set $T''=T'\setminus \{y\}$, $R''=R'\cap V(T'')$ (see Figure \ref{fig_radon_1}). Then $z\notin H_{T''}(R'')$ as otherwise we can partition $R'$ into 
$k$ parts,  $k-1$ of which contain $y$, and the last contains both $x_1$ and $z$. This partition is such that $y$ is in the convex hull of all parts, contradicting the fact that $R'$ is $k$-a.r..
Thus
\begin{equation*}
\r_k(T)=|R|=2k-1+|R''|< 3(k-1)+\r_k^*(T'',z)\le (k-1)(3+\r_2^*(T'',z)).
\end{equation*} 
The following claim completes the proof of Theorem \ref{partition} in Case 1\ref{case1b_partition}.
\begin{claim}\label{claim_radon}
$\r_2(T)\ge 3+\r_2^*(T'',z)$.
\end{claim}
\begin{proof}
Let $S''$ be an a.r.~set of $T''$ satisfying $z\notin H_{T''}(S'')$.
We shall show that $S=S''\cup \{x_1,x_2,x_3\}$ is a.r. thus proving the claim (note that $l\ge 3$, so $S$ is well defined).
Assume that we have a Radon partition $S=A\cup B$.
Without loss of generality $x_1,x_2\in A$.
Set $A'=(A\cap V(T'))\cup \{y\}$, $B'=B\cap V(T')$. 
Then
$H_{T}(A)\cap H_T(B)=H_{T'}(A')\cap H_{T'}(B')$.
Claim \ref{claim_radon} follows from the following claim.
\begin{claim}
Let $T$ be a tree, $v\in V(T)$ and $S$ an a.r.~set in $T$ satisfying $v\notin H_T(S)$.
Let $T_{v\leftarrow u}$ denote the tree obtained from $T$ by adding a new vertex $u$ and connecting it to $v$.
Then $S\cup \{u\}$ is a.r.~in $T_{v\leftarrow u}$.
\end{claim}
\begin{proof}
We prove the claim by induction on $|V(T)|$.
The claim is clear when $T$ has at most one vertex.
Let $T'=T_{v\leftarrow u}$ and $S=A\cup B$ a partition of $S$.
We show $H_{T'}(A\cup\{u\})\cap H_{T'}(B)=\emptyset$.

Let $v_1,\ldots,v_l$ be the neighbours of $v$ in $T$. 
Denote by $T_i$ the component of $v_i$ in $T\setminus \{v\}$ and $S_i=S\cap V(T_i)$, $A_i=A\cap V(T_i)$, $B_i=B\cap V(T_i)$.
If for every $i\in[l]$, $v_i\notin H_{T_i}(A_i)$, then $H_{T'}(A\cup \{u\})=H_{T'}(A)\cup \{u\}$ and $H_{T'}(B)=H_{T}(B)$.
Thus, as $S$ is a.r., $H_{T'}(A\cup\{u\})\cap H_{T'}(B)=H_T(A)\cap H_T(B)=\emptyset$.

So we can assume $v_1\in H_{T_1}(A_1)$.
As $v\notin H_T(S)$, this means that $v_i\notin H_{T_i}(S_i)$ for every $i\ge 2$.
As $S$ is a.r., $v_1\notin H_{T_1}(B_1)$.
Thus
\begin{align*}
&H_{T'}(A\cup\{u\})=H_{T_1}(A_1)\cup \{u\}\cup (\bigcup_{j\ge 2}H_{(T_i)_{v_i\leftarrow v}}(A_i\cup \{v\})) \\
&H_{T'}(B)=\bigcup_{i\ge 1}H_{T_i}(B_i)=H_{T_1}(B_1)\cup(\bigcup_{i\ge 2}H_{(T_i)_{v_i\leftarrow v}}B_i).
\end{align*}
Therefore
\begin{align*}
&H_{T'}(A\cup \{u\})\cap H_{T'}(B)=\bigcup_{i\ge 2}( H_{(T_i)_{v_i\leftarrow v}}(A_i\cup \{v\})\cap H_{(T_i)_{v_i\leftarrow v}}(B_i)   )
\end{align*}
The proof follows using the induction hypothesis with $T_i$, $i\ge 2$.
\end{proof}
\end{proof}
\item
$s\ge 2k$.\\
Similarly to Claim \ref{pairing}, we can conclude that the multi-set obtained by adding $k$ copies of $y$ to $R\cap V(T')$ is $k$-a.r.~, which is obviously a contradiction.

\end{enumerate}

\subsection{Case 2.}

In every longest path $v_1,\ldots,v_m$ of $T$, $\deg(v_2)=2$.
\setlength{\parindent}{0pt} 
\\

Fix a longest path $v_1,\ldots,v_m$ in $T$.
Denote $v_3=z$, and note that each of its neighbours other than $v_4$ is either an endvertex or has degree $2$ and is adjacent to an endvertex (by the choice of the longest path and the definition of Case 2). Let $y_1,\ldots,y_p$ be the neighbours of $z$ other than $v_4$ which have degree $2$, and $y_{p+1},\ldots,y_q$ the neighbours of $z$ other than $v_4$ which are endvertices.
Let $x_1,\ldots,x_p$ be the neighbours of $y_1,\ldots,y_p$ which are endvertices respectively (see Figure \ref{fig_radon2}).

\begin{figure}[ht]\centering
\subfloat[Case 2]{
\label{fig_radon2}
\begin{tikzpicture}
\tikzstyle{bel}=[below=0.05]
\draw[fill](0,0) circle [radius=0.125];
\draw[fill](-1,1.5) circle [radius=0.125];
\draw[fill](-1,.5) circle [radius=0.125];
\draw[fill](-1,-.2) circle [radius=0.125];
\draw[fill](-1,-1.2) circle [radius=0.125];
\draw[fill](-2,1.5) circle [radius=0.125];
\draw[fill](-2,.5) circle [radius=0.125];
\draw[fill](1,0) circle [radius=0.125];

\draw (0,0)--(1,0)--(1.8,.4);
\draw (1,0)--(1.8,-.7);
\draw (-2,1.5)--(-1,1.5)--(0,0);
\draw (-2,.5)--(-1,.5)--(0,0);
\draw (-1,-.2)--(0,0)--(-1,-1.2);

\draw[fill](-1,1.2) circle [radius=0.025];
\draw[fill](-1,1.0) circle [radius=0.025];
\draw[fill](-1,.8) circle [radius=0.025];

\draw[fill](-1,-.5) circle [radius=0.025];
\draw[fill](-1,-.7) circle [radius=0.025];
\draw[fill](-1,-.9) circle [radius=0.025];

\node[bel] at (1,0) {$v_4$};
\node[bel] at (0,0) {$z$};
\node[above left, yshift=-3, xshift=1] at (-1,1.5) {$y_1$};
\node[above left, yshift=-3, xshift=1] at (-1,0.5) {$y_p$};
\node[left] at (-1,-0.2) {$y_{p+1}$};
\node[left] at (-1,-1.2) {$y_q$};
\node[left] at (-2,0.5) {$x_p$};
\node[left] at (-2,1.5) {$x_1$};

\boldmath
\node at (0.1,0.8){$T'$};
\node at (1.1,0.6){$T''$};
\draw [rounded corners=25pt](2,-1.2)--(-.5,-1.2)--(-.5,1.2)--(2,1.2);
\draw [rounded corners=20pt](2,-.9)--(.5,-.9)--(.5,.9)--(2,.9);
\end{tikzpicture}
}
\hspace{.5in}
\subfloat[Case 2\ref{case2a_partition}]{
\label{fig_radon2a}
\begin{tikzpicture}
\tikzstyle{bel}=[below=0.05]
\draw[fill](0,0) circle [radius=0.125];
\draw[fill](1,0) circle [radius=0.125];
\draw[fill](2,0) circle [radius=0.125];
\draw[fill](3,0) circle [radius=0.125];
\draw (3.8,-.7)--(3,0)--(3.8,.4);
\draw (0,0)--(3,0);

\draw [rounded corners=25pt](4,-1)--(2.5,-1)--(2.5,1)--(4,1);
\node[bel]at(0,0){$x_1$};
\node[bel]at(1,0){$y_1$};
\node[bel]at(2,0){$z$};
\node[bel]at(3,0){$v_4$};
\boldmath
\node at (3.1,0.65){$T'$};

\end{tikzpicture}
}
\caption{Case 2}

\end{figure}

Denote by $s_i$, $i\in [p]$, the number of appearances of $x_i$ in $R$; $t_i$, $i\in [q]$, the number of appearances of $y_i$ in $R$ and $u$ the number of appearances of $z$ in $R$.
Let $t=t_1+\ldots+t_q$. As in the previous case, we conclude from the fact that $R$ is $k$-a.r.~that $t\le 2k-1$.
Consider the following three cases.
\begin{enumerate}[(a)]
\item\label{case2a_partition}
$q=1$.\\
Then $t_1+\min\{s_1,u\}\le k-1$ (otherwise obtain a $k$-Radon partition of $R$ by putting a copy of $y_1$ in $t_1$ sets, and a copy of $x_1$ and $z$ in the other $k-t_1$ sets).
Thus 
\begin{equation*}
s_1+t_1+u=t_1+\min\{s_1,u\}+\max\{s_1,u\}\le 2(k-1).
\end{equation*}
Let $T'=T\setminus\{x_1,y_1,z\}$, $R'=R\cap V(T')$ (see Figure \ref{fig_radon2a}). Then
\begin{equation*}
\r_k(T)=|R|\le|R'|+2(k-1)\le\r_k(T')+2(k-1)\le  (k-1)(\r_2(T')+2).
\end{equation*}
The proof of Theorem \ref{partition} in Case 2\ref{case2a_partition} follows from the following claim.
\begin{claim}
$\r_2(T)\ge \r_2(T')+2$.
\end{claim}
\begin{proof}
Let $S'$ be an a.r.~set in $T'$. Set $S=S\cup \{x_1,y_1\}$.
It is easy to verify that $S$ is a.r..
\end{proof}

\item
$t \le 2k-2$.\\
As before, set $\tau=\min\{t,2k-2-t\}$, $\sigma=(t-\tau)/2$ (then $t-\sigma\le k-1$).
Let $T'=T\setminus\{x_1,\ldots,x_p,y_1,\ldots,y_q\}$ (see Figure \ref{fig_radon2}) and let $R'$ be the multi-set obtained by adding $\sigma$ copies of $z$ to $R'\cap V(T')$.
Then as in Claim \ref{pairing}, $R'$ is $k$-a.r. and thus
\begin{equation*}
\r_k(T)=|R|=|R'|+s+t-\sigma\le \r_k(T')+(p+1)(k-1)\le (k-1)(\r_2(T')+p+1).
\end{equation*} 
The proof of Theorem \ref{partition} in this case follows from the following claim.
\begin{claim}
$\r_2(T)\ge\r_2(T')+p+1$.
\end{claim}
\begin{proof}
Let $S'$ be a.r.~in $T'$.
Set
\begin{align*}
S=\left\{
\begin{array}{ll}
S'\cup\{x_1,\ldots,x_p,y_1\}&z\notin S'\\
(S'\setminus \{z\})\cup\{x_1,\ldots,x_p,y_1,y_2\}&z\in S'
\end{array}
\right.
\end{align*}
Showing that $S$ is a.r.~is similar to the proof of Claim \ref{adding_leafs}.
\end{proof}

\item
$t=2k-1$.\\
Let $T'=T\setminus\{x_1,\ldots,x_p,y_1,\ldots,y_q\}$, $T''=T'\setminus\{z\}$ (see Figure \ref{fig_radon2}).
Let $R''=R\cap V(T'')$ and let $R'$ be the multi-set obtained by adding $k-1$ copies of $z$ to $R''$ and a copy of $y_1$. As in Case 1\ref{case1b_partition}, $R'$ is $k$-a.r.~in $T'$ and $R''$ is $k$-a.r.~in $T''$ with $v_4\notin H_{T''}(R'')$.
Thus
\begin{align}\label{last_equation}
&\r_k(T)=|R''|+s+2k-1\le 
\r_k^*(T'',v_4)+s+2k-1\le (k-1)\r_2^*(T'',v_4)+s+2k-1.
\end{align}
As $s_i\le k-1$ for every $i\in[p]$, $s\le (k-1)p$.
If $s<(k-1)p$, we obtain
\begin{equation*}
\r_k(T)\le(k-1)(\r_2^*(T'',v_4)+p+2).
\end{equation*}
And the proof of  Theorem \ref{partition} in this case follows from the claim below.
\begin{claim}
$\r_2(T)\ge \r_2(T'',v_4)+p+2$.
\end{claim}
\begin{proof}
If $S''$ is a.r.~in $T''$ with $v_4\notin S''$ then similarly to the proof of Claim \ref{claim_radon}, $S=S''\cup\{x_1,\ldots,x_p,y_1,y_2\}$ is a.r.~in $T$.
\end{proof}

Thus we may assume $s=(k-1)p$ i.e.~$s_1=\ldots=s_p=k-1$.
\begin{claim}\label{claim_pairing2}
$t_1=\ldots=t_p=0$.
\end{claim}
\begin{proof}
Assume otherwise, then without loss of generality $t_1\ge 1$. Let $\phi=k-t_1$.
Similarly to the proof of Claim \ref{pairing} we will show the existence of $a_1,\ldots,a_\phi$, $b_1,\ldots,b_\phi$ such that
\begin{itemize*}
\item
$a_j,b_j\in [2,q]$ and $a_j\neq b_j$ for every $j\in [\phi]$,
\item
$|\{j\in[\phi]\!: a_j=i\}|+|\{j\in[\phi]\!: b_j=i \}|\le t_i$ for every $i\in [2,q]$.
\end{itemize*}
This leads to a contradiction as we can then obtain a $k$-Radon partition of $R$ by putting a copy of $y_1$ in $t_1$ of the sets, and putting a copy of $x_1$ and a pair $y_{a_l},y_{b_l}$ in each of the other $k-t_1$ sets. $y_1$ will be in the intersection of the convex hulls of the sets (here we use the assumption that $s_1=k-1$ so this is indeed possible). 

If $t_i\le k-t_1-1$ for every $i\in [2,q]$, we proceed as in Claim \ref{pairing} to prove the existence of such sequences.
Otherwise, let $i_0$ be such that $t_{i_0}\ge k-t_1$.
Note that 
\begin{equation*}
\sum\limits_{j\neq 1,i_0}t_i=t-t_{i_0}-t_1\ge 2k-1-(k-1)-t_1=k-t_1.
\end{equation*}
Thus in this case we can choose $a_1=\ldots=a_\phi=i_0$ and $b_1,\ldots,b_\phi\in [2,q]\setminus\{i_0\}$ to satisfy the requirements.
\end{proof}
Using Claim \ref{claim_pairing2} it follows that $2k-1=t=t_{p+1}+\ldots+t_q$. As $t_i\le k-1$ for every $i\in [q]$, $q-p\ge 3$.
\begin{claim}
$\r_2(T)\ge \r_2^*(T'',v_4)+3+p$. 
\end{claim}
\begin{proof}
Let $S''$ be an a.r.~set in $T''$ with $v_4\notin H_{T''}(S'')$. Let $S=S''\cup\{x_1,\ldots,x_p,y_{p+1},y_{p+2},y_{p+3}\}$. It is easy to see that $S$ is a.r.~in $T$.
\end{proof}

Recalling inequality \ref{last_equation}, we obtain
\begin{equation*}
\r_k(T)\le s+2k-1+(k-1)\r_2^*(T'',v_4)\le(k-1)(p+3+\r_2^*(T'',v_4))\le (k-1)\r_2(T).
\end{equation*}
And the proof of Theorem \ref{partition} is complete.

\end{enumerate}

\end{proof}

\section{Proof of Theorem \ref{free}}\label{section_free}
We need the following definition for the proof of Theorem \ref{free}. Let $G$ be a graph, $v\in V(G)$. Define
\begin{align*}
\a ^*(T,v)\triangleq\max\{|A|\!: A\text{ is free and }x\notin A \}.
\end{align*}

\begin{proof}[ of Theorem \ref{free}]
We prove a stronger statement than what is claimed in this theorem. We shall show that for every tree $T$
\begin{itemize}
\item
$\r^*(T,v)\le 2\a^*(T,v)$ for every vertex $v\in V(T)$,
\item
$\r(T)\le 2\a(T)$.
\end{itemize}
We prove these statements by induction on $n=|V(T)|$.
Both statements are clear for $n\le 3$.

To prove the first statement, let $v\in V(T)$, and denote by $v_1,\ldots,v_l$ its neighbours. For every $i\in[l]$, let $T_i$ be the connected component of $v_i$ in $T\setminus \{v\}$. It is easy to see that
\begin{align*}
\a ^*(T,v)= \max_{j\in[k]}\{\sum\limits_{i\neq j}\a^*(T_i,v_i)+\a(T_j)  \}.
\end{align*}
Note the similarity to expression \ref{eqn_radon} from the previous section.
It thus follows  by induction that $\r^*(T,v)\le 2\a^*(T,v)$.

We now proceed to proving that $\r(T)\le 2\a(T)$.
Let $R$ be an a.r.~set of maximal size in $T$.
As in the proof of Theorem \ref{partition}, we can assume that $R$ contains all endvertices of $T$.
The \emph{brothers} of an endvertex $v$ are the endvertices in distance $2$ from $v$.
Then in particular, every endvertex has at most $2$ brothers, as no vertex of $T$ can have more than $3$ neighbours in $R$.

The following claim will be useful in the rest of the proof.
\begin{claim}\label{leafs_free}
Let $T$ be a tree. There exists a free set $A\subseteq V(T)$ of size $\a(T)$ satisfying that for every endvertex $v\in V(T)$ either $v\in A$ or one of its brothers is in $A$.
\end{claim}
\begin{proof}
Let $A$ be a free set in $T$ of maximal size, $v$ an endvertex in $T$ and $u$  its only neighbour. If $v\in A$ we are done. Otherwise, by the maximality of $A$, $A\cup\{v\}$ is not free. As $u$ is the only neighbour of $v$, there is a neighbour $w\neq v$ of $u$ which is contained in $A$.
If $w$ is an endvertex, we are done. Otherwise, set $A'=(A\backslash\{w\})\cup\{v\}$. Then $A'$ contains $v$ and is free of size $\a(T)$.
Continuing similarly will result in a free set of size $\a(T)$ with the property that for each endvertex either it or one of its brothers is in the set.
\end{proof}

We consider three cases concerning longest paths in $T$. Note that the theorem can be easily verified if the longest path in $T$ has at most $3$ vertices, thus we assume that a longest path in $T$ contains at least $4$ vertices. 
We devote a separate subsection for each case.
\subsection{Case 1.}

There is a longest path $v_1,\ldots,v_m$ such that the component of $v_4$ in $T\setminus\{v_5\}$ has no endvertex in distance $3$ from $v_4$ with brothers.
\\

We consider six cases.
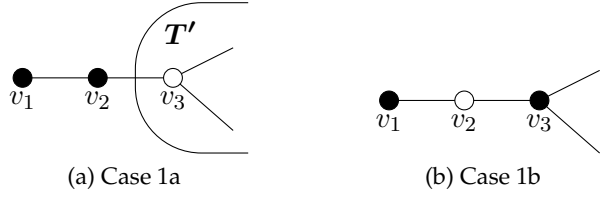
\begin{enumerate}[(a)]
\item\label{case1a}
$v_1,v_2\in R$.\\
Set $T'=T\setminus\{v_1,v_2\}$ (see Figure \ref{fig_free_1a}), $R'=R\cap V(T')$.

\begin{figure}
\centering
\subfloat[Case 1\ref{case1a}]{
\label{fig_free_1a} 
\begin{tikzpicture}
\tikzstyle{bel}=[below=0.05]

\draw (0,0)--(2,0)--(2.8,0.4);
\draw (2,0)--(2.8,-0.7);

\draw [fill=white] (2,0) circle [radius=0.125] ;
\draw [fill] (1,0) circle [radius=0.125];
\draw [fill] (0,0) circle [radius=0.125];

\node [bel] at (2,0) {$v_3$};
\node[bel] at (1,0) {$v_2$}; 
\node[bel] at (0,0) {$v_1$};

\draw [rounded corners=25pt](3,-1)--(1.5,-1)--(1.5,1)--(3,1);
\boldmath
\node at (2.1,0.6){$T'$};

\unboldmath

\end{tikzpicture}
}
\hspace{.5in}
\subfloat[Case 1\ref{case1b}]{
\label{fig_free_1b} 
\begin{tikzpicture}
\tikzstyle{bel}=[below=0.05]

\draw (0,0)--(2,0)--(2.8,0.4);
\draw (2,0)--(2.8,-0.7);

\draw [fill] (2,0) circle [radius=0.125] ;
\draw [fill=white] (1,0) circle [radius=0.125];
\draw [fill] (0,0) circle [radius=0.125];

\node [bel] at (2,0) {$v_3$};
\node[bel] at (1,0) {$v_2$}; 
\node[bel] at (0,0) {$v_1$};

\end{tikzpicture}
}
\\
\vspace{0.2in}
\small{In this figure and the following ones:}
\\ 
\small{a black vertex is in $R$, a white one is not in $R$,}
\\
\small{and for a grey vertex it is unknown if it is in $R$.}

\caption{Cases 1\ref{case1a}, 1\ref{case1b}}.
\end{figure}

$R'$ is a.r.~in $T'$ and $v_3\notin H_{T'}(R')$.
Thus, by induction,
\begin{equation*}
\r(T)=|R'|+2\le \r^*(T',v_3)+2\le 2(\a^*(T',v_3)+1).
\end{equation*}
Note that $\a^*(T',v_3)+1\le \a(T)$, because if $A'\subseteq V(T')\setminus\{v_3\}$ is free, then $A'\cup \{v_1\}$ is free in $T$.
Therefore $\r(T)\le 2\a(T)$ in this case.

\item\label{case1b}
$v_1,v_3\in R$ (see Figure \ref{fig_free_1b}). \\Set $R'=(R\backslash\{v_3\})\cup \{v_2\}$.
It is easy to see that $R'$ is a.r. and it follows from Case 1\ref{case1a} that $\r(T)\le 2\a(T)$.
\end{enumerate}

We can now assume that the above two cases do not occur.
Consider the neighbours of $v_3$ other than $v_4$.
Each such neighbour either is an endvertex, or has degree $2$ and its other neighbour is an endvertex (using the fact $v_1,\ldots,v_m$ is a longest path and that we are in Case 1).
Let $S_i$, $i\in\{1,2\}$, be the set of neighbours of $v_3$ other then $v_4$ with degree $i$. 
Note that $v_2\in S_2$, and by our previous assumptions: $S_1\subseteq R$, $S_2\cap R=\emptyset$. In particular $|S_1|\le 3$. Consider the remaining four cases.

\begin{enumerate}[(a)]
\setcounter{enumi}{2}
\item\label{case1c}
$|S_2|\ge 2$.\\
Let $T'$ be the component of $v_3$ in $T\setminus (S_2\setminus\{v_2\})$ (see Figure \ref{fig_free_1c}).
Then
\begin{equation*}
\r(T')\ge \r(T)-(|S_2|-1),
\end{equation*} 
as $R\cap V(T')$ is a.r.~in $T'$ and $R\setminus V(T')$ contains only the endvertices which are neighbours of vertices in $S_2\setminus\{v_2\}$. Furthermore
\begin{equation*}
\a(T)\ge \a(T')+(|S_2|-1).
\end{equation*}
To see this, let $A'\subseteq V(T')$ be a maximum sized free set in $T'$ containing $v_1$ (recall Claim \ref{leafs_free}). Then $v_3\notin A'$ and the set obtained by adding the endvertices which are neighbours of the vertices in $S_2$ to $A'$ is free in $T$. 
Hence, by induction,
\begin{equation*}
\r(T)\le \r(T')+|S_2|-1<2(\a(T')+|S_2|-1)\le2\a(T).
\end{equation*}
\end{enumerate}
We can now assume that $S_2=\{v_2\}$.

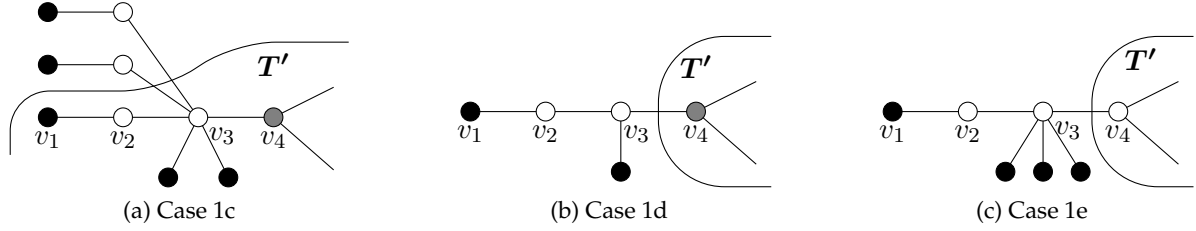
\begin{figure}[t]
\centering
\subfloat[Case 1\ref{case1c}]{
\label{fig_free_1c} 
\begin{tikzpicture}
\tikzstyle{bel}=[below=0.05]

\draw (0,0)--(3,0)--(3.8,0.4);
\draw (3,0)--(3.8,-0.7);
\draw (0,1.4)--(1,1.4)--(2,0)--(1,0.7)--(0,0.7);
\draw (2.4,-0.8)--(2,0)--(1.6,-.8);

\draw [fill=white] (2,0) circle [radius=0.125] ;
\draw [fill=white] (1,0) circle [radius=0.125];
\draw [fill] (0,0) circle [radius=0.125];
\draw [fill] (0,1.4) circle [radius=0.125];
\draw [fill=white] (1,1.4) circle [radius=0.125];
\draw [fill] (0,.7) circle [radius=0.125];
\draw [fill=white] (1,.7) circle [radius=0.125];
\draw [fill=gray] (3,0) circle [radius=0.125];
\draw [fill] (1.6,-.8) circle [radius=0.125];
\draw [fill] (2.4,-.8) circle [radius=0.125];

\node [below, xshift=9] at (2,0) {$v_3$};
\node[bel] at (1,0) {$v_2$}; 
\node[bel] at (0,0) {$v_1$};
\node[bel] at (3,0) {$v_4$};

\draw [rounded corners=15pt](-0.5,-.5)--(-.5,.35)--(1.5,.35)--(2.5,1)--(4,1);
\boldmath
\node at (3,0.7){$T'$};

\unboldmath

\end{tikzpicture}
}
\hspace{0.4in}
\subfloat[Case 1\ref{case1d}]{
\label{fig_free_1d} 
\begin{tikzpicture}
\tikzstyle{bel}=[below=0.05]

\draw (0,0)--(3,0)--(3.8,0.4);
\draw (3,0)--(3.8,-0.7);

\draw (2,0)--(2,-.8);

\draw [fill=white] (2,0) circle [radius=0.125] ;
\draw [fill=white] (1,0) circle [radius=0.125];
\draw [fill] (0,0) circle [radius=0.125];
\draw [fill=gray] (3,0) circle [radius=0.125];
\draw [fill] (2,-.8) circle [radius=0.125];

\node [bel, xshift=6] at (2,0) {$v_3$};
\node[bel] at (1,0) {$v_2$}; 
\node[bel] at (0,0) {$v_1$};
\node[bel] at (3,0){$v_4$};

\draw [rounded corners=25pt](4,-1)--(2.5,-1)--(2.5,1)--(4,1);
\boldmath
\node at (3,0.6){$T'$};

\unboldmath
\end{tikzpicture}
}
\hspace{.4in}
\subfloat[Case 1\ref{case1e}]{
\label{fig_free_1e}
\begin{tikzpicture}
\tikzstyle{bel}=[below=0.05]

\draw (0,0)--(3,0)--(3.8,0.4);
\draw (3,0)--(3.8,-0.7);

\draw (2,0)--(2,-.8);
\draw (2.5,-.8)--(2,0)--(1.5,-.8);

\draw [fill=white] (2,0) circle [radius=0.125] ;
\draw [fill=white] (1,0) circle [radius=0.125];
\draw [fill] (0,0) circle [radius=0.125];
\draw [fill=white] (3,0) circle [radius=0.125];
\draw [fill] (2,-.8) circle [radius=0.125];
\draw [fill] (2.5,-.8) circle [radius=0.125];
\draw [fill] (1.5,-.8) circle [radius=0.125];

\node [below, xshift=9.5] at (2,0) {$v_3$};
\node[bel] at (1,0) {$v_2$}; 
\node[bel] at (0,0) {$v_1$};
\node[bel] at (3,0){$v_4$};

\draw [rounded corners=25pt](4,-1)--(2.65,-1)--(2.65,1)--(4,1);
\boldmath
\node at (3.3,0.7){$T'$};
\unboldmath
\end{tikzpicture}
}
\caption{Case 1\ref{case1c}, 1\ref{case1d}, 1\ref{case1e}}
\end{figure}

\begin{enumerate}[(a)]
\setcounter{enumi}{3}
\item\label{case1d}
$|S_1|\le 1$.\\
Let $T'$ be the connected component of $v_4$ in $T\setminus \{v_3\}$ (see Figure \ref{fig_free_1d}).
Then
$
\r(T')\ge \r(T)-2
$,
as $R\cap V(T')$ is a.r.~in $T'$, and $v_2,v_3\notin R$ (otherwise consider Cases 1\ref{case1a},1\ref{case1b}).
Also
$
\a(T)\ge \a(T')+1
$,
because if $A'\subseteq V(T')$ is free in $T'$ then $A'\cup \{v_1\}$ is free in $T$.
We obtain
\begin{equation*}
\r(T)\le \r(T')+2\le 2(\a(T')+1)\le 2\a(T).
\end{equation*}
\item\label{case1e}
$|S_1|=3$.\\
Let $T'$ be as in the previous case (see Figure \ref{fig_free_1e}) and set $R'=R\cap V(T')$.
Then $v_4\notin H_{T'}(R')$ and $|R|=|R'|+4$.
Also $ \a(T)\ge\a^*(T,v_4)+2$, because if $A'\subseteq V(T')\setminus\{v_4\}$ is free, then $A'\cup\{v_1,v_2\}$ is free in $T$.
Hence
\begin{equation*}
\r(T)\le \r^*(T',v_4)+4\le 2(\a^*(T,v_4)+2)\le 2\a(T).
\end{equation*}

\item\label{case1f}
$|S_1|=2$.\\
Set $T'=T\setminus\{v_1,v_2\}$ (see Figure \ref{fig_free_1f1}). If $T'$ contains a free set of maximal size $A'$ such that $v_3\notin A'$, then $A'\cup\{v_1\}$ is free, so in this case $\a(T)\ge 1+\a(T')$ and 
\begin{equation*}
\r(T)\le 1+\r(T')< 2(1+\a(T'))\le 2\a(T).
\end{equation*}

\begin{figure}[t]
\centering
\subfloat[]{
\label{fig_free_1f1}
\begin{tikzpicture}[xscale=1]
\tikzstyle{bel}=[below=0.05]

\draw (0,0)--(3,0)--(3.8,0.4);
\draw (3,0)--(3.8,-0.7);

\draw (2.4,-0.8)--(2,0)--(1.6,-.8);

\draw [fill=white] (2,0) circle [radius=0.125] ;
\draw [fill=white] (1,0) circle [radius=0.125];
\draw [fill] (0,0) circle [radius=0.125];

\draw [fill=gray] (3,0) circle [radius=0.125];
\draw [fill] (1.6,-.8) circle [radius=0.125];
\draw [fill] (2.4,-.8) circle [radius=0.125];

\node [below, xshift=9] at (2,0) {$v_3$};
\node[bel] at (1,0) {$v_2$}; 
\node[bel] at (0,0) {$v_1$};
\node[bel] at (3,0) {$v_4$};

\draw [rounded corners=20pt](4,-1.2)--(1.3,-1.2)--(1.3,1)--(4,1);
\boldmath
\node at (1.9,0.7){$T'$};
\unboldmath

\end{tikzpicture}
}
\hspace{.5in}
\subfloat[]{
\label{fig_free_1f2}
\begin{tikzpicture}[xscale=1]
\tikzstyle{bel}=[below=0.05]

\draw (0,0)--(4,0)--(4.8,0.4);
\draw (4,0)--(4.8,-0.7);

\draw (2.4,-0.8)--(2,0)--(1.6,-.8);

\draw (0,-1.5)--(2,-1.5);

\draw (2.4,-2.3)--(2,-1.5)--(1.6,-2.3);

\draw [rounded corners=30pt](2,-1.5)--(3,-1.5)--(3,0);

\draw [fill=white] (2,0) circle [radius=0.125] ;
\draw [fill=white] (1,0) circle [radius=0.125];
\draw [fill] (0,0) circle [radius=0.125];

\draw [fill=gray] (3,0) circle [radius=0.125];
\draw [fill=gray] (4,0) circle [radius=0.125];
\draw [fill] (1.6,-.8) circle [radius=0.125];
\draw [fill] (2.4,-.8) circle [radius=0.125];

\draw [fill=white] (2,-1.5) circle [radius=0.125] ;
\draw [fill=white] (1,-1.5) circle [radius=0.125];
\draw [fill] (0,-1.5) circle [radius=0.125];

\draw [fill] (1.6,-2.3) circle [radius=0.125];
\draw [fill] (2.4,-2.3) circle [radius=0.125];

\node [below, xshift=9] at (2,0) {$v_3$};
\node[bel] at (1,0) {$v_2$}; 
\node[bel] at (0,0) {$v_1$};
\node[bel, xshift=6] at (3,0) {$v_4$};
\node[bel] at (4,0) {$v_5$};

\draw [rounded corners=20pt](5,-1)--(3.5,-1)--(3.5,1)--(5,1);
\draw [rounded corners=20pt] (-.5,-2)--(-.5,-1)--(1.5,-1)--(1.5,1.3)--(5,1.3);
\boldmath
\node at (2,.9){$T'$};
\node at (4.1,0.7){$T''$};
\unboldmath

\end{tikzpicture}
}
\caption{Case 1\ref{case1f}}
\end{figure}

Therefore we may assume that $v_3$ is contained in every maximum sized free set of $T'$.
Let $S$ be the set of neighbours of $v_4$ other than $v_5$ and for $v\in S$ let $T_v$ be the connected component of $v$ in $T\setminus\{v_4\}$.

We need the following claim.
\begin{claim}\label{claim_free_depth_two}
$T_v$ has depth $2$ as a tree rooted in $v$ for every $v\in S$. 
\end{claim}
\begin{proof}
Let $v\in S$, $v\neq v_3$ (note that the claim is clear if $v=v_3$).
By the choice of $v_1,\ldots,v_m$ as a longest path in $T$, $T_v$ has depth at most $2$.
We now show that $T_v$ has depth at least $2$, i.e.~$v$ has neighbours which are not endverties.
Let $A'$ be a free set of maximal size in $T'$.
If $v$ has no neighbour in $T_v$ which is not an endvertex, then $(A'\setminus\{v_3\})\cup\{v\}$ is also a free set of the same size in $T'$, contradicting our previous  assumption. 
\end{proof}

If for some $v\in S$, $T_v$ is not isomorphic to $T_{v_3}$ (as rooted trees at $v$,$v_3$ respectively), by changing the selected longest path to go through $v$ instead of $v_3$, we go back to one of the previous cases.
Thus we may assume that the trees $T_v$,  $v\in S$, are all isomorphic to $T_{v_3}$.

Let $T''$ be the component of $v_5$ in $T\bs\{v_4\}$ (see Figure \ref{fig_free_1f2}), $R''=R\cap V(T'')$.
Then $|R\bs R''|\le 3|S|+1$ as for each $v\in S$, $|R\cap V(T_v)|=3$ and possibly $v_4$ in $R$.
Note also that $\a(T)\ge \a(T'')+2|S|$, because the union of a free set of $T''$ with the endvertices in distance $3$ from $v_4$ and their neighbours is free in $T$.
Therefore
\begin{equation*}
\r(T)\le \r(T'')+3|S|+1\le 2\a(T'')+4|S|\le2\a(T).
\end{equation*}

\end{enumerate}

\subsection{Case 2.}

Case 1 does not hold, and there exists a longest path $v_1,\ldots,v_m$ such that the component of $v_4$ in $T\setminus\{v_5\}$ has no endvertices in distance $4$ from $v_4$ with more than one brother.
\\

Choose the longest path such that $v_1$ has exactly one brother $v_1'$.
Then $v_1,v_1'\in R$ and as $R$ is a.r., $v_2\notin R$.
We consider the neighbours of $v_3$ other than $v_4$. Note that they can be of degrees $1$, $2$ or $3$ only and that if they have degree $2$ or $3$ the other neighbours are endvertices.
Let $S_i$, $i\in\{1,2,3\}$, be the set of neighbours of $v_3$ other than $v_4$ with degree $i$. Consider the following six cases.

\begin{enumerate}[(a)]
\item\label{case2a}
$S_2\neq \emptyset$. \\
Let $T'$ be the component of $v_3$ in $T\setminus S_2$ (i.e.~remove all neighbours of $v_3$ of degree $2$, see Figure \ref{fig_free_2a}).
Then, by induction
\begin{equation*}
\r(T)\le \r(T')+2|S_2|\le 2(\a(T')+|S_2|)\le 2\a(T).
\end{equation*}
The last inequality follows from the fact that a maximum free set in $T'$ can be assumed to contain $v_1$ (see Claim \ref{leafs_free}), so it does not contain $v_3$ and we can add the $|S_2|$ endvertices that were discarded to obtain a free set in $T$.
\end{enumerate}

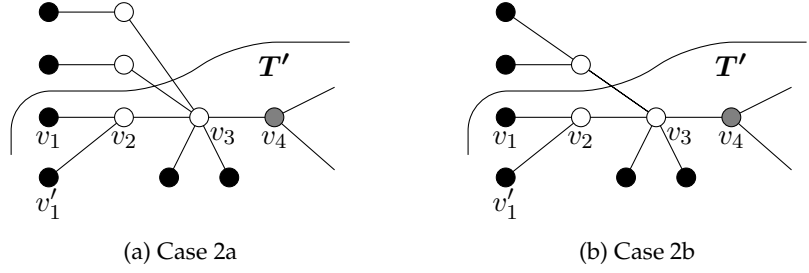
\begin{figure}[t]\centering
\subfloat[Case 2\ref{case2a}]{
\label{fig_free_2a}
\begin{tikzpicture}
\tikzstyle{bel}=[below=0.05]

\draw (0,0)--(3,0)--(3.8,0.4);
\draw (3,0)--(3.8,-0.7);
\draw (0,1.4)--(1,1.4)--(2,0)--(1,0.7)--(0,0.7);
\draw (2.4,-0.8)--(2,0)--(1.6,-.8);
\draw (1,0)--(0,-0.8);

\draw [fill] (0,-.8) circle [radius=.125];
\draw [fill=white] (2,0) circle [radius=0.125] ;
\draw [fill=white] (1,0) circle [radius=0.125];
\draw [fill] (0,0) circle [radius=0.125];
\draw [fill] (0,1.4) circle [radius=0.125];
\draw [fill=white] (1,1.4) circle [radius=0.125];
\draw [fill] (0,.7) circle [radius=0.125];
\draw [fill=white] (1,.7) circle [radius=0.125];
\draw [fill=gray] (3,0) circle [radius=0.125];
\draw [fill] (1.6,-.8) circle [radius=0.125];
\draw [fill] (2.4,-.8) circle [radius=0.125];

\node [below, xshift=9] at (2,0) {$v_3$};
\node[bel] at (1,0) {$v_2$}; 
\node[bel] at (0,0) {$v_1$};
\node[bel] at (3,0) {$v_4$};
\node[bel,yshift=1] at (0,-.8){$v_1'$};

\draw [rounded corners=15pt](-0.5,-.5)--(-.5,.35)--(1.5,.35)--(2.5,1)--(4,1);
\boldmath
\node at (3,0.7){$T'$};
\unboldmath

\end{tikzpicture}
}
\hspace{.5in}
\subfloat[Case 2\ref{case2b}]{
\label{fig_free_2b}
\begin{tikzpicture}
\tikzstyle{bel}=[below=0.05]

\draw (0,0)--(3,0)--(3.8,0.4);
\draw (3,0)--(3.8,-0.7);
\draw (0,1.4)--(2,0)--(1,0.7)--(0,0.7);
\draw (2.4,-0.8)--(2,0)--(1.6,-.8);
\draw (1,0)--(0,-0.8);

\draw [fill] (0,-.8) circle [radius=.125];
\draw [fill=white] (2,0) circle [radius=0.125] ;
\draw [fill=white] (1,0) circle [radius=0.125];
\draw [fill] (0,0) circle [radius=0.125];
\draw [fill] (0,1.4) circle [radius=0.125];

\draw [fill] (0,.7) circle [radius=0.125];
\draw [fill=white] (1,.7) circle [radius=0.125];
\draw [fill=gray] (3,0) circle [radius=0.125];
\draw [fill] (1.6,-.8) circle [radius=0.125];
\draw [fill] (2.4,-.8) circle [radius=0.125];

\node [below, xshift=9] at (2,0) {$v_3$};
\node[bel] at (1,0) {$v_2$}; 
\node[bel] at (0,0) {$v_1$};
\node[bel] at (3,0) {$v_4$};
\node[bel,yshift=1] at (0,-.8){$v_1'$};

\draw [rounded corners=15pt](-0.5,-.5)--(-.5,.35)--(1.5,.35)--(2.5,1)--(4,1);
\boldmath
\node at (3,0.7){$T'$};
\unboldmath

\end{tikzpicture}
}
\caption{Cases 2\ref{case2a}, 2\ref{case2b}}

\end{figure}

We now assume $S_2=\emptyset$.
\begin{enumerate}[(a)]
\setcounter{enumi}{1}
\item\label{case2b}
$|S_3|\ge 2$.\\
Set $T'$ to be the component of $v_3$ in $T\setminus(S_3\setminus\{v_3\})$ (see Figure \ref{fig_free_2b}).
$R$ contains $2(|S_3|-1)$ of the discarded vertices, thus, as in the previous case,
\begin{equation*}
\r(T)\le \r(T')+2(|S_3|-1)\le 2(\a(T')+(|S_3|-1))\le 2\a(T).
\end{equation*}
\end{enumerate}

Hence we can assume $|S_3|=1$.
Note that $|S_1|\le 2$, since $R$ is a.r..
\begin{enumerate}[(a)]
\setcounter{enumi}{2}
\item\label{case2c}
$|S_1|=2$.\\
Set $T'$ the component of $v_4$ in $T\setminus\{v_3\}$ (see Figure \ref{fig_free_2c}).
Then
\begin{equation*}
\r(T)\le 4+\r^*(T',v_4)\le 2(2+\a^*(T',v_4))\le 2\a(T),
\end{equation*}
since $v_4\notin H_{T'}(R\cap V(T'))$ ($R$ is a.r.) and the union of a free set in $V(T')\setminus\{v_4\}$ with $\{v_1,v_2\}$ remains free.

\begin{figure}[t] \centering
\subfloat[Case 2\ref{case2c}]{
\label{fig_free_2c}
\begin{tikzpicture}
\tikzstyle{bel}=[below=0.05]

\draw (1,0)--(3,0)--(3.8,0.4);
\draw (3,0)--(3.8,-0.7);
\draw (.2,.4)--(1,0)--(.2,-.4);
\draw (2.4,-0.8)--(2,0)--(1.6,-.8);

\draw [fill] (0.2,-.4) circle [radius=.125];
\draw [fill=white] (2,0) circle [radius=0.125] ;
\draw [fill=white] (1,0) circle [radius=0.125];
\draw [fill] (0.2,0.4) circle [radius=0.125];

\draw [fill=white] (3,0) circle [radius=0.125];
\draw [fill] (1.6,-.8) circle [radius=0.125];
\draw [fill] (2.4,-.8) circle [radius=0.125];

\node [below, xshift=9] at (2,0) {$v_3$};
\node[bel] at (1,0) {$v_2$}; 
\node[bel] at (.2,0.4) {$v_1$};
\node[bel] at (3,0) {$v_4$};
\node[bel,yshift=1] at (0.2,-.4){$v_1'$};

\draw [rounded corners=20pt](4,-1)--(2.6,-1)--(2.6,1)--(4,1);
\boldmath
\node at (3.1,0.7){$T'$};
\unboldmath

\end{tikzpicture}
}
\hspace{.5in}
\subfloat[Case 2\ref{case2d}]{
\label{fig_free_2d}
\begin{tikzpicture}
\tikzstyle{bel}=[below=0.05]

\draw (1,0)--(3,0)--(3.8,0.4);
\draw (3,0)--(3.8,-0.7);
\draw (.2,.4)--(1,0)--(.2,-.4);

\draw [fill] (0.2,-.4) circle [radius=.125];
\draw [fill] (2,0) circle [radius=0.125] ;
\draw [fill=white] (1,0) circle [radius=0.125];
\draw [fill] (0.2,0.4) circle [radius=0.125];

\draw [fill=white] (3,0) circle [radius=0.125];

\node [below, xshift=9] at (2,0) {$v_3$};
\node[bel] at (1,0) {$v_2$}; 
\node[bel] at (.2,0.4) {$v_1$};
\node[bel] at (3,0) {$v_4$};
\node[bel,yshift=1] at (0.2,-.4){$v_1'$};

\draw [rounded corners=20pt](4,-1)--(2.6,-1)--(2.6,1)--(4,1);
\boldmath
\node at (3.1,0.7){$T'$};
\unboldmath

\end{tikzpicture}
}
\hspace{.5in}
\subfloat[Case 2\ref{case2e}]{
\label{fig_free_2e}
\begin{tikzpicture}
\tikzstyle{bel}=[below=0.05]

\draw (1,0)--(3,0)--(3.8,0.4);
\draw (3,0)--(3.8,-0.7);
\draw (.2,.4)--(1,0)--(.2,-.4);

\draw [fill] (0.2,-.4) circle [radius=.125];
\draw [fill=white] (2,0) circle [radius=0.125] ;
\draw [fill=white] (1,0) circle [radius=0.125];
\draw [fill] (0.2,0.4) circle [radius=0.125];

\draw [fill=gray] (3,0) circle [radius=0.125];

\node [below, xshift=9] at (2,0) {$v_3$};
\node[bel] at (1,0) {$v_2$}; 
\node[bel] at (.2,0.4) {$v_1$};
\node[bel] at (3,0) {$v_4$};
\node[bel,yshift=1] at (0.2,-.4){$v_1'$};

\draw [rounded corners=20pt](4,-1)--(2.6,-1)--(2.6,1)--(4,1);
\boldmath
\node at (3.1,0.7){$T'$};
\unboldmath

\end{tikzpicture}
}
\caption{Cases 2\ref{case2c}, 2\ref{case2d}, 2\ref{case2e}}
\end{figure}
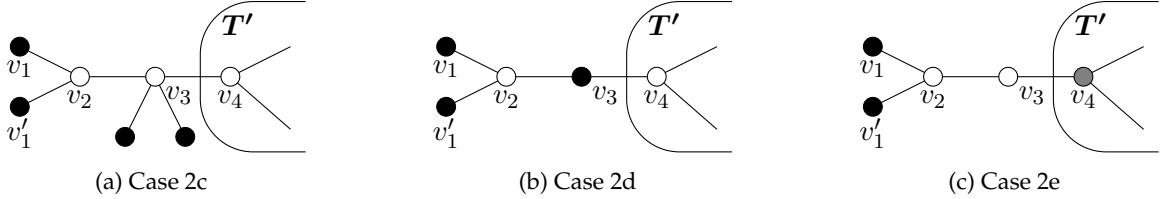

\item\label{case2d}
$S_1=\emptyset$ and $v_3\in R$.\\
Choose $T'$ as in the previous case (see Figure \ref{fig_free_2d}). Again $v_4\notin H_{T'}(R\cap V(T'))$ and similarly
\begin{equation*}
\r(T)\le 3+\r^*(T',v_4)< 2(2+\a^*(T',v_4))\le 2\a(T).
\end{equation*}

\item\label{case2e}
$S_1=\emptyset$ and $v_3\notin R$.\\
Again set $T'$ as before (see Figure \ref{fig_free_2e}).
Here $R$ contains only $2$ of the discarded vertices and we can add $v_1$ to any free set of $T'$ to obtain a free set of $T$. Thus
\begin{equation*}
\r(T)\le\r(T')+2\le 2(\a(T')+1)\le 2\a(T).
\end{equation*}

\item\label{case2f}
All previous cases do not hold, i.e.~
$|S_1|=|S_3|=1$ and $S_2=\emptyset$.\\
Set $T'=T\setminus\{v_1,v_1',v_2\}$ (see Figure \ref{fig_free_2f1}).
If $T'$ contains a maximum free set $A'$ with $v_3\notin A'$, then $A'\cup\{v_1\}$ is free in $T$ and thus
\begin{equation*}
\r(T)\le \r(T')+2\le 2(1+\a(T'))\le 2\a(T).
\end{equation*}

\begin{figure}[t] \centering
\subfloat[]{
\label{fig_free_2f1}
\begin{tikzpicture}
\tikzstyle{bel}=[below=0.05]

\draw (1,0)--(3,0)--(3.8,0.4);
\draw (3,0)--(3.8,-0.7);
\draw (.2,.4)--(1,0)--(.2,-.4);
\draw (2,0)--(2,-.8);

\draw [fill] (0.2,-.4) circle [radius=.125];
\draw [fill=white] (2,0) circle [radius=0.125] ;
\draw [fill=white] (1,0) circle [radius=0.125];
\draw [fill] (0.2,0.4) circle [radius=0.125];

\draw [fill=gray] (3,0) circle [radius=0.125];

\draw [fill] (2,-0.8) circle [radius=0.125];

\node [bel, xshift=6.5] at (2,0) {$v_3$};
\node[bel] at (1,0) {$v_2$}; 
\node[bel] at (.2,0.4) {$v_1$};
\node[bel] at (3,0) {$v_4$};
\node[bel,yshift=1] at (0.2,-.4){$v_1'$};

\draw [rounded corners=20pt](4,-1.2)--(1.5,-1.2)--(1.5,1)--(4,1);
\boldmath
\node at (2.1,0.7){$T'$};
\unboldmath

\end{tikzpicture}
}
\hspace{.5in}
\subfloat[]{
\label{fig_free_2f2}
\begin{tikzpicture}
\tikzstyle{bel}=[below=0.05]

\draw (1,0)--(4,0)--(4.8,0.4);
\draw (4,0)--(4.8,-0.7);
\draw (.2,.4)--(1,0)--(.2,-.4);
\draw (2,-0.8)--(2,0);

\draw (0,-1.5)--(2,-1.5);

\draw (2.4,-2.3)--(2,-1.5)--(1.6,-2.3);

\draw [rounded corners=30pt](2,-1.5)--(3,-1.5)--(3,0);

\draw [fill=white] (2,0) circle [radius=0.125] ;
\draw [fill=white] (1,0) circle [radius=0.125];
\draw [fill] (0.2,0.4) circle [radius=0.125];
\draw [fill] (0.2,-0.4) circle [radius=0.125];

\draw [fill=gray] (3,0) circle [radius=0.125];
\draw [fill=gray] (4,0) circle [radius=0.125];
\draw [fill] (2,-.8) circle [radius=0.125];

\draw [fill=white] (2,-1.5) circle [radius=0.125] ;
\draw [fill=white] (1,-1.5) circle [radius=0.125];
\draw [fill] (0,-1.5) circle [radius=0.125];

\draw [fill] (1.6,-2.3) circle [radius=0.125];
\draw [fill] (2.4,-2.3) circle [radius=0.125];

\node [bel, xshift=6.5] at (2,0) {$v_3$};
\node[bel] at (1,0) {$v_2$}; 
\node[bel] at (0.2,0.4) {$v_1$};
\node[bel] at (0.2,-0.4) {$v_1'$};
\node[bel, xshift=6] at (3,0) {$v_4$};
\node[bel] at (4,0) {$v_5$};

\draw [rounded corners=20pt](5,-1)--(3.5,-1)--(3.5,1)--(5,1);
\draw [rounded corners=20pt] (-.5,-2)--(-.5,-1.1)--(1.5,-1.1)--(1.5,1.3)--(5,1.3);
\boldmath
\node at (2.1,1){$T'$};
\node at (4.1,0.7){$T''$};
\unboldmath

\end{tikzpicture}
}
\caption{Case 2\ref{case2f}}
\end{figure}

Therefore, we may assume that every maximum free set of $T'$ contains $v_3$.
As in Case 1\ref{case1f}, let $S$ be the set of neighbours of $v_4$ different from $v_5$ and for $v\in S$ define $T_v$ to be the component of $v$ in $T\setminus\{v_4\}$.
Similarly to Claim \ref{claim_free_depth_two}, $T_v$ has depth $2$ as a tree rooted at $v$ for every $v\in S$.
If $T_v$ is not isomorphic to $T_{v_3}$ or to the graph in Case 1\ref{case1f} (as rooted trees), by changing the longest path to go through $v$, we can continue as before. (Note that in all but the present case and Case 1\ref{case1f}, we did not consider other neighbours of $v_4$).

Set $T''$ to be the component of $v_5$ in $T\setminus\{v_4\}$ (see Figure \ref{fig_free_2f2}).
$R$ contains three vertices of $T_v$ for every $v\in S$ and possibly it contains $v_4$ as well. Also $\a(T)\ge\a(T'')+2|S|$, as we can two vertices from each $T_v$ to a free set of $T''$ to obtain a free set. Thus
\begin{equation*}
\r(T)\le \r(T'')+3|S|+1\le 2(\a(T'')+2|S|)\le 2\a(T).
\end{equation*}

\end{enumerate}

\subsection{Case 3.}
For every choice of a longest path $v_1,\ldots,v_m$ the connected component of $v_4$ in $T\setminus\{v_5\}$ has an endvertex with $2$ brothers in distance $3$ from $v_4$.
\\

We choose the longest path such that $v_1$ has $2$ brothers $v_1'$ and $v_1''$.
Then $v_1,v_1',v_1''\in R$, so $v_2,v_3\notin R$.

Similarly to Cases 2\ref{case2a},2\ref{case2b}, we can assume that all the neighbours of $v_3$ are endvertices (except for maybe $v_4$) having one or two endvertices as neighbours. By the choice of $v_1,\ldots,v_m$ as a longest path, the neighbours of $v_3$ other than $v_4$ are either endvertices or have degree $4$ and are neighbours to $3$ endvertices.
Thus, as $R$ is a.r., $v_3$ can have degree $2$ or $3$ only.
Set $T'$ to be the component of $v_4$ in $T\setminus \{v_3\}$.
We consider seven possible cases.

\begin{enumerate}[(a)]
\item\label{case3a}
$v_3$ has degree $3$ with the only neighbour other than $v_2$ and $v_4$ being an endvertex (see Figure \ref{fig_free_3a}).\\
Then $v_4\notin H_{T'}(R\cap V(T'))$.
Note that $\a(T)\ge\a^*(T',v_4)+2$, as we can add $v_1$ and $v_2$ to a free set in $V(T')\setminus\{v_4\}$ to obtain a free set.
Thus
\begin{equation*}
\r(T)\le \r^*(T',v_4)+4\le 2(\a^*(T',v_4)+2)\le 2\a(T).
\end{equation*}

\begin{figure}[t] \centering
\subfloat[Case 3\ref{case3a}]{
\label{fig_free_3a}
\begin{tikzpicture}
\tikzstyle{bel}=[below=0.05]

\draw (0.2,0)--(3,0)--(3.8,0.4);
\draw (3,0)--(3.8,-0.7);
\draw (.2,.6)--(1,0)--(.2,-.6);
\draw (2,0)--(2,-.8);

\draw [fill=white] (2,0) circle [radius=0.125] ;
\draw [fill=white] (1,0) circle [radius=0.125];
\draw [fill] (0.2,0.6) circle [radius=0.125];
\draw [fill] (0.2,-.6) circle [radius=.125];
\draw [fill] (0.2,0) circle [radius=.125];
\draw [fill=white] (3,0) circle [radius=0.125];

\draw [fill] (2,-0.8) circle [radius=0.125];

\node [bel, xshift=6.5] at (2,0) {$v_3$};
\node[bel] at (1,0) {$v_2$}; 
\node[left] at (.2,0.6) {$v_1$};
\node[left] at (.2,0) {$v_1'$};
\node[left] at (.2,-0.6) {$v_1''$};
\node[bel] at (3,0) {$v_4$};

\draw [rounded corners=20pt](4,-1)--(2.5,-1)--(2.5,1)--(4,1);
\boldmath
\node at (3.1,0.7){$T'$};
\unboldmath

\end{tikzpicture}
}
\hspace{.3in}
\subfloat[Case 3\ref{case3b}]{
\label{fig_free_3b}
\begin{tikzpicture}
\tikzstyle{bel}=[below=0.05]

\draw (0.2,0)--(3,0)--(3.8,0.4);
\draw (3,0)--(3.8,-0.7);
\draw (.2,.6)--(1,0)--(.2,-.6);
\draw (2,0)--(2,-1.8);
\draw (2.6,-1.8)--(2,-1)--(1.4,-1.8);

\draw [fill=white] (2,0) circle [radius=0.125] ;
\draw [fill=white] (1,0) circle [radius=0.125];
\draw [fill] (0.2,0.6) circle [radius=0.125];
\draw [fill] (0.2,-.6) circle [radius=.125];
\draw [fill] (0.2,0) circle [radius=.125];
\draw [fill=white] (3,0) circle [radius=0.125];

\draw [fill=white] (2,-1)circle [radius=0.125];
\draw [fill] (2,-1.8) circle [radius=0.125];
\draw [fill] (2.6,-1.8) circle [radius=0.125];
\draw [fill] (1.4,-1.8) circle [radius=0.125];

\node [bel, xshift=6.5] at (2,0) {$v_3$};
\node[bel] at (1,0) {$v_2$}; 
\node[left] at (.2,0.6) {$v_1$};
\node[left] at (.2,0) {$v_1'$};
\node[left] at (.2,-0.6) {$v_1''$};
\node[bel] at (3,0) {$v_4$};
\node[right,xshift=1] at (2,-1) {$u$};

\draw [rounded corners=20pt](4,-1)--(2.5,-1)--(2.5,1)--(4,1);
\boldmath
\node at (3.1,0.7){$T'$};
\unboldmath

\end{tikzpicture}
}
\hspace{.3in}
\subfloat[Case 3\ref{case3c}]{
\label{fig_free_3c}
\begin{tikzpicture}
\tikzstyle{bel}=[below=0.05]

\draw (0.2,0)--(3,0)--(3.8,0.4);
\draw (3,0)--(3.8,-0.7);
\draw (.2,.6)--(1,0)--(.2,-.6);

\draw [fill=white] (2,0) circle [radius=0.125] ;
\draw [fill=white] (1,0) circle [radius=0.125];
\draw [fill] (0.2,0.6) circle [radius=0.125];
\draw [fill] (0.2,-.6) circle [radius=.125];
\draw [fill] (0.2,0) circle [radius=.125];
\draw [fill=gray] (3,0) circle [radius=0.125];

\node [bel] at (2,0) {$v_3$};
\node[bel] at (1,0) {$v_2$}; 
\node[left] at (.2,0.6) {$v_1$};
\node[left] at (.2,0) {$v_1'$};
\node[left] at (.2,-0.6) {$v_1''$};
\node[bel] at (3,0) {$v_4$};

\draw [rounded corners=20pt](4,-1)--(2.5,-1)--(2.5,1)--(4,1);
\boldmath
\node at (3.1,0.7){$T'$};
\unboldmath

\end{tikzpicture}
}
\caption{Cases 3\ref{case3a}, 3\ref{case3b}, 3\ref{case3c}}
\end{figure}
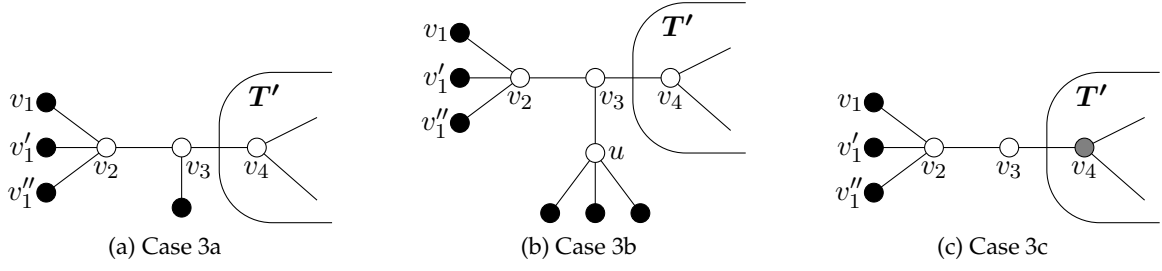

\item\label{case3b}
$v_3$ has degree $3$ with the only neighbour other than $v_2$ and $v_4$, $u$, having three neighbours which are endvertices (see Figure \ref{fig_free_3b}).\\
Then again $v_4\notin H_{T'}(R\cap V(T'))$, and $\a(T)\ge 3+\a^*(T',v_4)$, as we can add $v_1$, $v_2$ and an endvertex which is a neighbour of $u$ to a free set of $T'$. Thus
\begin{equation*}
\r(T)\le 6+\r^*(T',v_4)\le 2(3+\a^*(T',v_4))\le 2\a^*(T).
\end{equation*}
\end{enumerate}

In the remaining cases we assume that $v_3$ has degree $2$.
Let  $T''$ be the component of $v_5$ in $T\setminus\{v_4\}$.
\begin{enumerate}[(a)]
\setcounter{enumi}{2}
\item\label{case3c}
$T'$ has a maximum free set $A'$ with $v_4\notin A'$ (see Figure \ref{fig_free_3c}).
\\
Then $A'\cup \{v_1,v_2\}$ is free and 
\begin{equation*}
\r(T)\le 3+\r(T')<2(\a(T')+2)\le 2\a(T).
\end{equation*}
\end{enumerate}

We may  now assume that every maximum free set of $T'$ contains $v_4$.
Let $S$ be the set of neighbours of $v_4$ other than $v_3$ and $v_5$.
\begin{claim}\label{claim_free_S_endvertices}
The vertices in $S$ are endvertices in $T$.
\end{claim} 
\begin{proof}
Let $v\in S$, and $T_v$ the component of $v$ in $T\setminus\{v_4\}$.
Then $T_v$ has depth at most $2$ as a tree rooted in $v$ (by the choice of $v_1,\ldots,v_m$ as a longest path).
Let $A'$ be a free set of maximal size in $T'$, then $v_4\in A'$. If $v$ has a neighbour in $T_v$, $u$, it is either an endevertex, or all of its neighbours except for $v$ are endvertices. Then $(A'\setminus\{v_4\})\cup \{u\}$ is free in $T'$, a contradiction.
Thus $v$ has no neighbours in $T_v$, i.e.~it is an endvertex in $T$.
\end{proof}
Clearly, the claim implies $|S|\le 3$.

\begin{enumerate}[(a)]
\setcounter{enumi}{3}
\item\label{case3d}
$S=\emptyset$ (see Figure \ref{fig_free_3d}). 
\\
Then $|R\cap(V(T)\setminus V(T''))|\le 4$ and  we can add $v_1,v_2$ to a free set of $T''$. Thus
\begin{equation*}
\r(T)\le 4+ \r(T'')\le 2(2+\a(T''))\le 2\a(T).
\end{equation*}
\end{enumerate}

\begin{figure}[t] \centering
\subfloat[Case 3\ref{case3d}]{
\label{fig_free_3d}
\begin{tikzpicture}
\tikzstyle{bel}=[below=0.05]

\draw (0.2,0)--(4,0)--(4.8,0.4);
\draw (4,0)--(4.8,-0.7);
\draw (.2,.6)--(1,0)--(.2,-.6);

\draw [fill=white] (2,0) circle [radius=0.125] ;
\draw [fill=white] (1,0) circle [radius=0.125];
\draw [fill] (0.2,0.6) circle [radius=0.125];
\draw [fill] (0.2,-.6) circle [radius=.125];
\draw [fill] (0.2,0) circle [radius=.125];
\draw [fill=gray] (3,0) circle [radius=0.125];
\draw [fill=gray] (4,0) circle [radius=0.125];

\node [bel] at (2,0) {$v_3$};
\node[bel] at (1,0) {$v_2$}; 
\node[left] at (.2,0.6) {$v_1$};
\node[left] at (.2,0) {$v_1'$};
\node[left] at (.2,-0.6) {$v_1''$};
\node[bel] at (3,0) {$v_4$};
\node[bel] at (4,0) {$v_5$};

\draw [rounded corners=20pt](5,-.9)--(3.5,-.9)--(3.5,.9)--(5,.9);
\draw [rounded corners=20pt](5,-1.2)--(2.5,-1.2)--(2.5,1.2)--(5,1.2);
\boldmath
\node at (3.1,.9){$T'$};
\node at (4.1,0.6){$T''$};
\unboldmath

\end{tikzpicture}
}
\hspace{.5in}
\subfloat[Case 3\ref{case3e}]{
\label{fig_free_3e}
\begin{tikzpicture}
\tikzstyle{bel}=[below=0.05]

\draw (0.2,0)--(4,0)--(4.8,0.4);
\draw (4,0)--(4.8,-0.7);
\draw (.2,.6)--(1,0)--(.2,-.6);
\draw (3.6,-.8)--(3,0)--(2.4,-.8);
\draw (3,-.8)--(3,0);

\draw [fill=white] (2,0) circle [radius=0.125] ;
\draw [fill=white] (1,0) circle [radius=0.125];
\draw [fill] (0.2,0.6) circle [radius=0.125];
\draw [fill] (0.2,-.6) circle [radius=.125];
\draw [fill] (0.2,0) circle [radius=.125];
\draw [fill=gray] (3,0) circle [radius=0.125];
\draw [fill=gray] (4,0) circle [radius=0.125];
\draw [fill] (3,-.8) circle [radius=0.125];
\draw [fill] (2.4,-.8) circle [radius=0.125];
\draw [fill] (3.6,-.8) circle [radius=0.125];

\node [bel] at (2,0) {$v_3$};
\node[bel] at (1,0) {$v_2$}; 
\node[left] at (.2,0.6) {$v_1$};
\node[left] at (.2,0) {$v_1'$};
\node[left] at (.2,-0.6) {$v_1''$};
\node[below right, xshift=2] at (3,0) {$v_4$};
\node[bel] at (4,0) {$v_5$};

\draw [rounded corners=20pt](5,-.9)--(3.7,-.9)--(3.7,.9)--(5,.9);
\draw [rounded corners=20pt](5,1.2)--(2.6,1.2)--(2.1,-1.2);
\boldmath
\node at (4.3,0.6){$T''$};
\node at (3,.8){$T'$};
\unboldmath

\end{tikzpicture}
}
\caption{Cases 3\ref{case3d}, 3\ref{case3e}}
\end{figure}
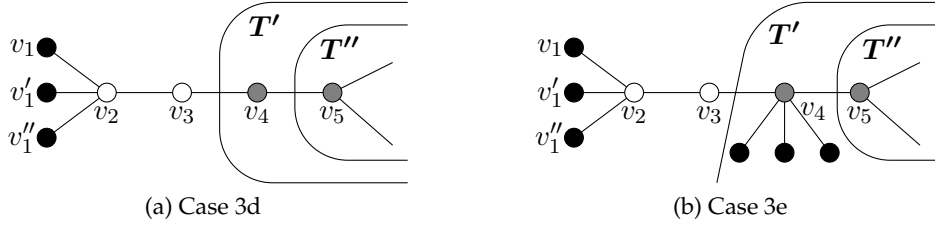

We can assume now that $S\neq \emptyset$.
\begin{enumerate}[(a)]
\setcounter{enumi}{4}
\item\label{case3e}
There is a free set $A''$ of maximal size in $T''$ with $v_5\notin A''$ (see Figure \ref{fig_free_3e}).
\\
Then $|R\cap (V(T)\setminus V(T''))|\le 6$ and the union of $A''$ with $v_1$, $v_2$ and an endvertex from $S$ is free, thus
\begin{equation*}
\r(T)\le 6+\r(T'')\le 2(3+\a(T''))\le 2\a(T).
\end{equation*}
\end{enumerate}
Thus we can assume that every maximal free set in $T''$ contains $v_5$.

\begin{enumerate}[(a)]
\setcounter{enumi}{5}
\item\label{case3f}
$v_5$ has degree $2$ in $T$.\\
Let $T'''$ be the component of $v_6$ in $T\setminus \{v_5\}$ (see Figure \ref{fig_free_3f}). Then, as in the previous case,
\begin{equation*}
\r(T)\le 6+\r(T''')\le 2(3+\a(T'''))\le 2\a(T).
\end{equation*}

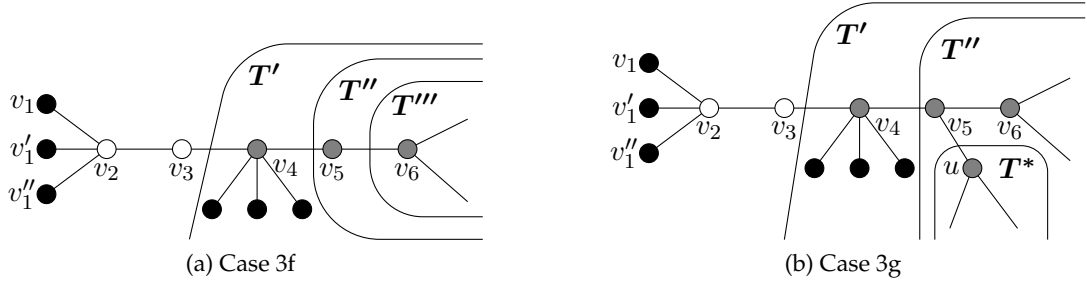
\begin{figure}[t] \centering
\subfloat[Case 3\ref{case3f}]{
\label{fig_free_3f}
\begin{tikzpicture}
\tikzstyle{bel}=[below=0.05]

\draw (0.2,0)--(5,0)--(5.8,0.4);
\draw (5,0)--(5.8,-0.7);
\draw (.2,.6)--(1,0)--(.2,-.6);
\draw (3.6,-.8)--(3,0)--(2.4,-.8);
\draw (3,-.8)--(3,0);

\draw [fill=white] (2,0) circle [radius=0.125] ;
\draw [fill=white] (1,0) circle [radius=0.125];
\draw [fill] (0.2,0.6) circle [radius=0.125];
\draw [fill] (0.2,-.6) circle [radius=.125];
\draw [fill] (0.2,0) circle [radius=.125];
\draw [fill=gray] (3,0) circle [radius=0.125];
\draw [fill=gray] (4,0) circle [radius=0.125];
\draw [fill=gray] (5,0) circle [radius=0.125];
\draw [fill] (3,-.8) circle [radius=0.125];
\draw [fill] (2.4,-.8) circle [radius=0.125];
\draw [fill] (3.6,-.8) circle [radius=0.125];

\node [bel] at (2,0) {$v_3$};
\node[bel] at (1,0) {$v_2$}; 
\node[left] at (.2,0.6) {$v_1$};
\node[left] at (.2,0) {$v_1'$};
\node[left] at (.2,-0.6) {$v_1''$};
\node[below right, xshift=2] at (3,0) {$v_4$};
\node[bel] at (4,0) {$v_5$};
\node[bel] at (5,0) {$v_6$};

\draw [rounded corners=25pt](6,-1.2)--(3.75,-1.2)--(3.75,1.2)--(6,1.2);
\draw [rounded corners=20pt](6,-.9)--(4.5,-.9)--(4.5,.9)--(6,.9);
\draw [rounded corners=20pt](6,1.4)--(2.7,1.4)--(2.1,-1.2);
\boldmath
\node at (4.35,0.8){$T''$};
\node at (5.1,.6){$T'''$};
\node at (3.1,1){$T'$};
\unboldmath

\end{tikzpicture}
}
\hspace{.5in}
\subfloat[Case 3\ref{case3g}]{
\label{fig_free_3g}
\begin{tikzpicture}
\tikzstyle{bel}=[below=0.05]

\draw (0.2,0)--(5,0)--(5.8,0.4);
\draw (5,0)--(5.8,-0.7);
\draw (.2,.6)--(1,0)--(.2,-.6);
\draw (3.6,-.8)--(3,0)--(2.4,-.8);
\draw (3,-.8)--(3,0);
\draw (4,0)--(4.5,-.8);
\draw (4.2,-1.6)--(4.5,-.8)--(5.1,-1.6);

\draw [fill=white] (2,0) circle [radius=0.125] ;
\draw [fill=white] (1,0) circle [radius=0.125];
\draw [fill] (0.2,0.6) circle [radius=0.125];
\draw [fill] (0.2,-.6) circle [radius=.125];
\draw [fill] (0.2,0) circle [radius=.125];
\draw [fill=gray] (3,0) circle [radius=0.125];
\draw [fill=gray] (4,0) circle [radius=0.125];
\draw [fill=gray] (5,0) circle [radius=0.125];
\draw [fill] (3,-.8) circle [radius=0.125];
\draw [fill] (2.4,-.8) circle [radius=0.125];
\draw [fill] (3.6,-.8) circle [radius=0.125];
\draw [fill=gray] (4.5,-.8) circle [radius=0.125];

\node [bel] at (2,0) {$v_3$};
\node[bel] at (1,0) {$v_2$}; 
\node[left] at (.2,0.6) {$v_1$};
\node[left] at (.2,0) {$v_1'$};
\node[left] at (.2,-0.6) {$v_1''$};
\node[below right, xshift=2] at (3,0) {$v_4$};
\node[below right, xshift=1] at (4,0) {$v_5$};
\node[bel] at (5,0) {$v_6$};
\node[left, xshift=-1] at (4.5,-.8) {$u$};

\draw [rounded corners=20pt](3.8,-1.75)--(3.8,1.2)--(6,1.2);
\draw [rounded corners=20pt](6,1.4)--(2.5,1.4)--(2,-1.75);
\draw [rounded corners=10pt] (4,-1.75)--(4,-.5)--(5.5,-.5)--(5.5,-1.75);

\boldmath
\node at (4.35,0.8){$T''$};
\node at (2.9,1){$T'$};
\node at (5.1,-.8){$T^*$};
\unboldmath

\end{tikzpicture}
}
\caption{Cases 3\ref{case3f}, 3\ref{case3g}}
\end{figure}

\item\label{case3g}
$v_5$ has a neighbour $u\neq v_4,v_6$. \\
Let $T^*$ be the component of $u$ in $T\setminus\{v_5\}$ (see Figure \ref{fig_free_3g}).
The following claim can be proved similarly to the proofs of Claims \ref{claim_free_depth_two}, \ref{claim_free_S_endvertices}, using the above assumptions.
\begin{claim}
$T^*$ has depth $3$ as a tree rooted in $u$.
\end{claim}

By considering a longest path going through $u$ instead of $v_4$, we can assume that the component of $u$ in $T\setminus \{v_5\}$ satisfies the same conditions as the component of $v_4$.
However, in this case $T''$ has an endvertex in distance $2$ from $v_5$, a contradiction to the assumption that every maximum free set of $T''$ contains $v_5$.

\end{enumerate}

\end{proof}

The following example shows that Theorem  \ref{free} is sharp. 

\begin{figure}[!h]\centering
\begin{tikzpicture}
\tikzstyle{bel}=[below=0.05]

\draw(0,0)--(7,0);
\draw(9,0)--(12,0);
\draw[fill](7.25,0) circle [radius=0.025];
\draw[fill](7.5,0) circle [radius=0.025];
\draw[fill](7.75,0) circle [radius=0.025];
\draw[fill](8,0) circle [radius=0.025];
\draw[fill](8.25,0) circle [radius=0.025];
\draw[fill](8.5,0) circle [radius=0.025];
\draw[fill](8.75,0) circle [radius=0.025];

\draw[fill](0,0) circle [radius=0.125];
\draw[fill](0,-1) circle [radius=0.125];
\draw[fill](1,-2) circle [radius=0.125];
\draw[fill](-1,-2) circle [radius=0.125];
\draw[fill](1,-3) circle [radius=0.125];
\draw[fill](.5,-3) circle [radius=0.125];
\draw[fill](1.5,-3) circle [radius=0.125];
\draw[fill](-1,-3) circle [radius=0.125];
\draw[fill](-.5,-3) circle [radius=0.125];
\draw[fill](-1.5,-3) circle [radius=0.125];

\draw (0,0)--(0,-1)--(-1,-2)--(-0.5,-3);
\draw (0,-1)--(1,-2)--(1.5,-3);
\draw (1,-3)--(1,-2)--(.5,-3);
\draw (-1,-3)--(-1,-2)--(-1.5,-3);

\node[above=0.05] at (0,0) {$v_1$};
\node[right=0.05, yshift=-1.5] at (0,-1) {$u_1$};
\node[right=0.05, yshift=-1.5] at (-1,-2) {$w_1$};
\node[right=0.05, yshift=-1.5] at (-0.5,-3) {$x_1$};
\node[right=0.05, yshift=-1.5] at (1.5,-3) {$y_1$};

\draw[fill](4,0) circle [radius=0.125];
\draw[fill](4,-1) circle [radius=0.125];
\draw[fill](5,-2) circle [radius=0.125];
\draw[fill](3,-2) circle [radius=0.125];
\draw[fill](5,-3) circle [radius=0.125];
\draw[fill](4.5,-3) circle [radius=0.125];
\draw[fill](5.5,-3) circle [radius=0.125];
\draw[fill](3,-3) circle [radius=0.125];
\draw[fill](3.5,-3) circle [radius=0.125];
\draw[fill](2.5,-3) circle [radius=0.125];

\draw (4,0)--(4,-1)--(3,-2)--(3.5,-3);
\draw (4,-1)--(5,-2)--(5.5,-3);
\draw (5,-3)--(5,-2)--(4.5,-3);
\draw (3,-3)--(3,-2)--(2.5,-3);

\node[above=0.05] at (4,0) {$v_2$};
\node[right=0.05, yshift=-1.5] at (4,-1) {$u_2$};
\node[right=0.05, yshift=-1.5] at (3,-2) {$w_2$};
\node[right=0.05, yshift=-1.5] at (3.5,-3) {$x_2$};
\node[right=0.05, yshift=-1.5] at (5.5,-3) {$y_2$};

\draw[fill](12,0) circle [radius=0.125];
\draw[fill](12,-1) circle [radius=0.125];
\draw[fill](13,-2) circle [radius=0.125];
\draw[fill](11,-2) circle [radius=0.125];
\draw[fill](13,-3) circle [radius=0.125];
\draw[fill](12.5,-3) circle [radius=0.125];
\draw[fill](13.5,-3) circle [radius=0.125];
\draw[fill](11,-3) circle [radius=0.125];
\draw[fill](11.5,-3) circle [radius=0.125];
\draw[fill](10.5,-3) circle [radius=0.125];

\draw (12,0)--(12,-1)--(11,-2)--(11.5,-3);
\draw (12,-1)--(13,-2)--(13.5,-3);
\draw (13,-3)--(13,-2)--(12.5,-3);
\draw (11,-3)--(11,-2)--(10.5,-3);

\node[above=0.05] at (12,0) {$v_m$};
\node[right=0.05, yshift=-1.5] at (12,-1) {$u_m$};
\node[right=0.05, yshift=-1.5] at (11,-2) {$w_m$};
\node[right=0.05, yshift=-1.5] at (11.5,-3) {$x_m$};
\node[right=0.05, yshift=-1.5] at (13.5,-3) {$y_m$};
\end{tikzpicture}
\caption{Sharpness of Theorem \ref{free}}
\label{fig_free}
\end{figure}
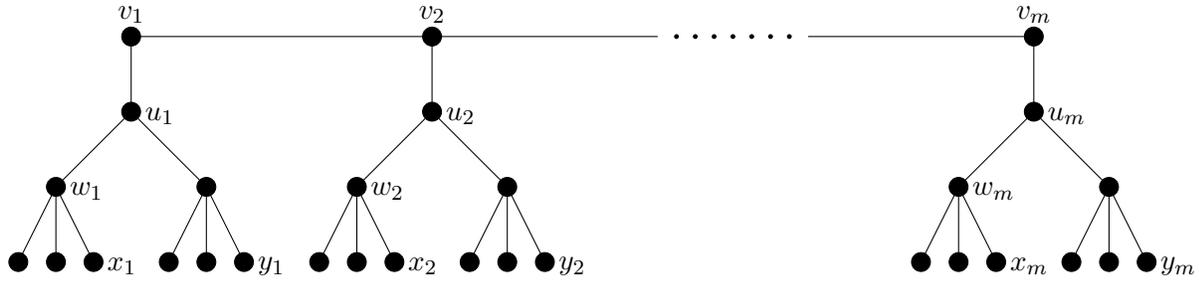

This is sequence of trees $T_m$, $m\ge 1$, with $10m$ vertices.
$\r(T)\ge 6m$ (the set of all endvertices is a.r.).
Let $A$ be a free set of $T_m$ with maximal size. We can assume that $A$ contains the endvertices $x_1,\ldots,x_m,y_1,\ldots,y_m$.
Thus $u_1,\ldots,u_m\notin A$.
Also, $A$ contains at most one of the $3$ neighbours of $u_i$ for each $i\in[m]$.
Hence $\{x_1,\ldots,x_m\}\cup\{y_1,\ldots,y_m\}\cup\{w_1,\ldots,w_m\}$ is a free set of maximal size, so $\a(T)=3m$.
By Theorem \ref{free}, $\r(T)\le 2\a(T)=6m$. Thus $\r(T)=6m=2\a(T)$.

\section{Concluding Remarks}\label{section_conclusion}
In this paper we proved two results about the Radon number for $P_3$-convexity in graphs. 
It may be interesting to consider these problems for general graphs. 
Regarding Theorem \ref{partition}, it is still an open problem to determine whether Eckhoff's conjecture holds for $P_3$-convexity in all graphs. 
We showed that the inequality $\r(G)\le 2\a(G)$ from Theorem \ref{free}, does not hold for all graphs $G$, but it may still be the case that a similar but weaker inequality holds in general.
Furthermore, for both results, it would  be interesting to characterize the trees for which the results hold with equality.

\bibliography{bib}
\bibliographystyle{plain}
\end{document}